\documentclass[10pt,a4paper]{amsart}

\usepackage[T1]{fontenc}
\usepackage[english]{babel}
\usepackage[a4paper,top=3cm,bottom=3.5cm,left=2cm,right=2cm]{geometry}

\usepackage{amsmath,amssymb,amsthm}
\usepackage{mathtools}
\usepackage{enumitem}
\usepackage{hyperref}
\usepackage{autonum}

\numberwithin{equation}{section}

\usepackage{aliascnt}

\theoremstyle{plain}
\newtheorem{theorem}{Theorem}[section]

\newaliascnt{lemma}{theorem}
\newtheorem{lemma}[lemma]{Lemma}
\aliascntresetthe{lemma}

\newaliascnt{proposition}{theorem}
\newtheorem{proposition}[proposition]{Proposition}
\aliascntresetthe{proposition}

\newaliascnt{corollary}{theorem}
\newtheorem{corollary}[corollary]{Corollary}
\aliascntresetthe{corollary}

\newtheorem{theoremA}{Theorem}

\newtheorem{theoremB}{Theorem }

\theoremstyle{definition}
\newaliascnt{definition}{theorem}
\newtheorem{definition}[definition]{Definition}
\aliascntresetthe{definition}

\theoremstyle{remark}
\newaliascnt{remark}{theorem}
\newtheorem{remark}[remark]{Remark}
\aliascntresetthe{remark}

\newaliascnt{notation}{theorem}

\aliascntresetthe{notation}

\usepackage[nameinlink,capitalise,sort]{cleveref}

\crefname{equation}{}{}
\crefname{enumi}{}{}

\crefname{theorem}{Theorem}{Theorem}
\crefname{lemma}{Lemma}{Lemma}
\crefname{proposition}{Proposition}{Proposition}
\crefname{corollary}{Corollary}{Corollary}
\crefname{definition}{Definition}{Definition}
\crefname{remark}{Remark}{Remark}
\crefname{notation}{Notation}{Notation}
\crefname{enumi}{}{}

\crefname{theoremA}{Theorem}{Theorem}
\crefname{theoremB}{Theorem}{Theorem}

\newcommand{\R}{\mathbb{R}}

\newcommand{\T}{\mathbb{T}}
\newcommand{\Z}{\mathbb{Z}}

\newcommand{\eps}{\varepsilon}

\DeclareMathOperator{\loc}{loc}

\DeclareMathOperator{\supp}{supp}
\DeclareMathOperator{\Div}{div}

\newcommand{\dd}{\,\mathrm{d}}

\title[$L^2(\R^2)$ Well-Posedness and Logarithmic Lipschitz Regularity for the Density Patch Problem]
{$L^2(\R^2)$ Well-Posedness and Logarithmic Lipschitz Regularity for the Density Patch Problem}

\date{}

\author[A.~Violini]{Alessandro~Violini}
\address[A.~Violini]{Universit\"at Basel, Department Mathematik und Informatik, Spiegelgasse 1, 4051 Basel, Switzerland. }
\email{alessandro.violini@unibas.ch}

\subjclass[2020]{76D03, 76D05, 46E35, 35Q30}
\keywords{Inhomogeneous Navier--Stokes equations; density patch problem; vacuum; uniqueness; log-Lipschitz regularity.}

\begin{document}

\begin{abstract}
We study the density patch problem for the two-dimensional inhomogeneous incompressible Navier--Stokes system with vacuum, for initial data consisting of a Lipschitz density patch and a divergence-free velocity field in $L^2(\mathbb{R}^2)$.

We establish uniqueness of solutions at the natural energy level, thereby concluding the global well-posedness for $L^2$ data. 

Furthermore, we prove a log-Lipschitz estimate for the velocity field, extending the classical result of Chemin-Lerner to the inhomogeneous setting. As a consequence, the associated flow belongs to $L^\infty_t C_x^{1-\eps}$ for any $\eps \in (0,1)$, ensuring that the patch boundary remains a continuous curve of Hausdorff dimension $1$, thus preserving its initial dimension for all time. 

\end{abstract}
\maketitle

\section{Introduction}

In this work, we address Lions' density patch problem \cite[p.~34]{Lions1996}. 
Specifically, we consider an initial density $\rho_0 = \mathbf{1}_D$ given by the indicator function of a Lipschitz domain $D$, which represents a patch of a homogeneous incompressible fluid surrounded by vacuum. 
For such an initial configuration, it is known that there exists at least one global weak solution where the density remains a patch $\rho(t)=\mathbf{1}_{D(t)}$ for all $t \ge 0$, maintaining a constant volume ($|D(t)| = |D|$). 
The fundamental question raised by Lions is whether the geometric regularity of the initial domain $D$ is preserved by the domain $D(t)$ during the time evolution.

As observed by Lions, this problem can be naturally formulated in terms of the inhomogeneous incompressible Navier--Stokes system posed on $\R_+ \times \R^2$, namely
\begin{equation}\label{eq:INS}
\left\{
\begin{aligned}
\partial_t(\rho u) + \Div(\rho u \otimes u) + \nabla P &= \Delta u, \\
\partial_t \rho + \Div(\rho u) &= 0, \\
\Div u &= 0,
\end{aligned}
\right.
\end{equation}
where $\rho=\rho(t,x)\ge 0$ denotes the density, $u=u(t,x)\in\R^2$ the velocity field, and $P=P(t,x)$ the pressure. For simplicity, the viscosity coefficient has been normalized to~$1$.

The system is supplemented with the initial data
\begin{equation}\label{eq:initial_data}
\rho(0,x)=\rho_0(x)=\mathbf{1}_D(x), 
\qquad 
u(0,x)=u_0(x) \in L^2_\sigma(\R^2),
\end{equation}
where $D \subset \R^2$ is a bounded smooth domain and $L^2_\sigma(\R^2)$ denotes the space of weakly divergence-free vector fields in $L^2(\R^2)$.

Under the assumptions in \eqref{eq:initial_data}, Lions proved in~\cite{Lions1996} the existence of a global Leray--Hopf weak solution, namely a distributional solution of \eqref{eq:INS} satisfying the energy inequality
\begin{equation}\label{eq:energy}
\frac12 \int_{\R^2} \rho |u(t)|^2 \,\dd x
+ \int_0^t \!\!\int_{\R^2} |\nabla u|^2 \,\dd x\,\dd s
\le \frac12 \int_{\R^2} \rho_0 |u_0|^2 \,\dd x
\end{equation}
for almost every $t\ge 0$.

In the first part of this work, we investigate the problem of uniqueness of these solutions. 
To put our result into perspective, let us first recall the situation in the nondegenerate case of strictly positive density,
\begin{equation}\label{eq:initial_data1}
0 < c < \rho_0(x) < C, 
\qquad 
u_0 \in L^2_\sigma(\R^2).
\end{equation}
where the well-posedness theory has been extensively studied over the last two decades. Early contributions established global well-posedness and uniqueness under additional regularity assumptions on the initial velocity; see, for instance,
\cite{Danchin2003,Abidi2007,AbidiGuiZhang2012,DanchinMucha2012,PaicuZhangZhang2013,Danchin2025}.

A major breakthrough was achieved by Hao, Shao, Wei, and Zhang~\cite{HaoShaoWeiZhang2025}, who proved global well-posedness and uniqueness in two space dimensions at the natural energy level $u_0 \in L_\sigma^2$. This result was subsequently extended to the full Leray--Hopf class in~\cite{Skondric2025}. As a consequence, in the absence of vacuum, the uniqueness theory is by now essentially well understood.

The situation is markedly different in the presence of vacuum. Indeed, the degeneracy of the momentum equation and the lack of integrability of the velocity field in vacuum regions introduce additional difficulties.

A fundamental contribution in this direction was obtained by Danchin and Mucha~\cite{DanchinMucha2019} (see also~\cite{DanchinMuchaPiasecki2024}), who proved global well-posedness for general nonnegative densities and initial data $u_0 \in H^1(\T^2)$. The restriction to the torus was later removed in~\cite{PrangeTan2023}, where uniqueness was established for $u_0 \in H^1(\R^2)$ under additional geometric assumptions on the initial density, including the patch configuration~\eqref{eq:initial_data}. Further contributions in this direction can be found in~\cite{GancedoGarcia2018,HaoShaoWeiZhang2026}.

More recently, uniqueness at critical regularity was established in~\cite{SkondricViolini2026}, where global well-posedness for initial data of the form~\eqref{eq:initial_data} was proved under the additional assumption
\begin{align}
u_0 \in \dot{B}^0_{2,1}(\R^2).
\end{align}

In the present work, we extend global well-posedness to the natural energy space $L^2$. 
To state our main result, we first introduce the time-weighted energy functionals
\begin{align}
\label{eq:Aidef}
A_0(t,u)
&:= \|\sqrt{\rho}\,u(t)\|_{L^2}^2
+ \int_0^t \|\nabla u(s)\|_{L^2}^2 \,\dd s,
\\
A_1(t,u)
&:= t\|\nabla u(t)\|_{L^2}^2 
+ \int_0^t s\,
\|\nabla^2 u(s),\nabla P(s),\sqrt{\rho}\dot u(s)\|_{L^2}^2
\,\dd s,
\\
A_2(t,u)
&:= t^2 \|\sqrt{\rho}\dot u(t),\nabla^2 u(t)\|_{L^2}^2
+ \int_0^t s^2 \|\nabla \dot u(s)\|_{L^2}^2 \,\dd s,
\\
A_3(t,u)
&:= t^3 \|\nabla \dot u(t)\|_{L^2}^2
+ \int_0^t s^3 \|\nabla^2 \dot u(s)\|_{L^2}^2 \,\dd s,
\end{align}
where
\begin{align}
\dot u := \partial_t u + (u \cdot \nabla) u
\end{align}
denotes the material derivative. We now introduce the class of solutions for which well-posedness is established (for an overview of these solutions, we refer to \cite{CrinBaratDeNittiSkondricViolini2025}).

\begin{definition}\label{def:immstrong}
A Leray--Hopf solution $(\rho, u)$ is called an \emph{immediately strong solution} if
\begin{equation}\label{eq:ImmStrongEst}
    \sup_{i=0,1,2,3} \; \sup_{t>0} A_i(t, u) =: C_u < \infty.
\end{equation}
\end{definition}

Our first main result is the following uniqueness theorem.

\begin{theoremA}\label{thm:A}
Let $(\rho_0,u_0)$ satisfy
\begin{equation}\label{eq:initial_data2}
\rho_0(x)=\mathbf{1}_D(x), 
\qquad 
u_0 \in L^2_\sigma(\R^2),
\end{equation}
where $D \subset \R^2$ is a bounded Lipschitz domain. Then there exists at most one immediately strong solution to~\eqref{eq:INS} with initial data $(\rho_0,u_0)$.
\end{theoremA}

Combined with the existence theory established in~\cite{SkondricViolini2026}, this yields global well-posedness in the natural energy class. For convenience, we recall the corresponding existence result.

\begin{theorem}[{\cite[Theorem 1.2, Section 3]{SkondricViolini2026}}]\label{thm:existL2}
Let $(\rho_0,u_0)$ be as in~\eqref{eq:initial_data2}. Then there exists an immediately strong solution $(\rho,u)$ to~\eqref{eq:INS}. This solution satisfies 
\begin{enumerate}[label=\roman*)]
\item $\rho u \in C([0,\infty);L^2(\R^2))$;
\item $\rho \in C([0,\infty);L^p_{\loc}(\R^2))$ for every $p<\infty$;
\item the energy identity~\eqref{eq:energy} holds with equality.
\end{enumerate}

Moreover, any immediately strong solution $(\bar\rho,\bar u)$ with initial data~\eqref{eq:initial_data2} admits an associated flow map $\bar X(t,x)$ satisfying the following parabolic displacement\footnote{The parabolic displacement holds true for any immediately strong solution with initial data \eqref{eq:initial_data2}, not only for the one constructed in \cite[Theorem 1.2]{SkondricViolini2026}.}
\begin{equation}\label{eq:dispBound}
|\bar X(t,x)-x|
\le C\sqrt{t}
\end{equation}
for almost every $x \in D$ and almost every $t \in (0,T)$, where the constant $C>0$ depends only on $C_{\bar u}, T$ and $D$.
\end{theorem}

In the second part of this work, we investigate the geometric evolution of the density patch. 
Since the density is transported by the flow associated with the velocity field $u$, namely the map $X(t,x)$ solving
\begin{align}
\frac{d}{dt}X(t,x)=u(t,X(t,x)),
\qquad 
X(0,x)=x,
\end{align}
one has
\begin{align}
\rho(t,x)=\mathbf{1}_{D(t)}(x),
\qquad 
D(t):=X(t,D).
\end{align}
Therefore, the regularity of the evolving patch $D(t)$ is entirely determined by the regularity of the flow map~$X$.

The first global propagation result in this direction was obtained by Danchin and Mucha~\cite{DanchinMucha2019}, who proved that if $u_0 \in H^1(\T^2)$, then the associated flow satisfies
\begin{align}
X \in C^{1,\gamma},
\qquad 
\gamma\in(0,1),
\end{align}
and consequently the $C^{1,\gamma}$ regularity of the patch is preserved for all times. This result was later extended in~\cite{GancedoGarciaJuarezLunaVelasco2025} to initial data $u_0 \in H^s(\R^2)$ with $s>\gamma$, and was recently improved in~\cite{SkondricViolini2026}, where the endpoint space $u_0 \in \dot B^\gamma_{2,1}(\R^2)$ was reached.

At the critical level $\gamma=0$, it was further shown in~\cite{SkondricViolini2026} that if $u_0 \in \dot B^0_{2,1}(\R^2)$ and $D$ is a bounded Lipschitz domain, then the corresponding velocity field satisfies
\begin{equation}\label{eq:Lip}
\int_0^\infty
\|\nabla u(t)\|_{L^\infty(\R^2)}
\,\dd t
<\infty.
\end{equation}
As a consequence, the associated flow is uniformly Lipschitz in time, and therefore the Lipschitz regularity of the patch is propagated globally.

On the other hand, if one only assumes $u_0 \in L^2(\R^2)$, the estimate~\eqref{eq:Lip} is no longer expected to hold. Indeed, already for the two-dimensional heat equation, one can construct solutions with initial data in $L^2$ for which
\begin{align}
\nabla u
\notin
L^1_{\loc}([0,\infty);L^\infty(\R^2)).
\end{align}
Nevertheless, solutions to the heat equation with $L^2$ initial data still satisfy a weaker log-Lipschitz estimate of the form
\begin{equation}\label{eq:logLip}
|u(t,x)-u(t,y)|
\le
\gamma(t)\,
|x-y|\,
\bigl(-\log|x-y|\bigr)^{1-\eta},
\end{equation}
for $|x-y|<1$, where $\eta\in(0,\tfrac12)$ and $\gamma \in L^1_{\loc}([0,\infty))$.

For the homogeneous Navier--Stokes equations, such an estimate was established by Chemin and Lerner in~\cite{CheminLerner1995}; see also~\cite[Theorem~5.43]{BCD2011}. More precisely, they proved~\eqref{eq:logLip} for Leray--Hopf solutions arising from divergence-free initial data $u_0 \in L^2(\R^2)$.

Our next result extends this log-Lipschitz regularity to the inhomogeneous setting, both in the presence and absence of vacuum.

\begin{theoremB}\label{thm:B}
Assume one of the following two scenarios:
\begin{enumerate}[label=\roman*)]
\item $(\rho,u)$ is the unique Leray--Hopf solution corresponding to strictly positive density as in~\eqref{eq:initial_data1};
\item $(\rho,u)$ is the unique immediately strong solution corresponding to~\eqref{eq:initial_data2}.
\end{enumerate}
Then, for every $\eta\in(0,\tfrac12)$, there exists
\begin{align}
\gamma \in L^1_{\loc}([0,\infty))
\end{align}
such that, for almost every $t>0$ and every $x,y\in\R^2$ with $0<|x-y|<1$,
\begin{align}\label{eq:logLip2}
|u(t,x)-u(t,y)|
\le
\gamma(t)\,
|x-y|\,
\bigl(-\log|x-y|\bigr)^{1-\eta}.
\end{align}
\end{theoremB}

In contrast to~\cite{CheminLerner1995} (see also \cite[Theorem 5.43]{BCD2011}), our proof does not rely on frequency localization, but is instead based on an atomic decomposition of the initial data. Although we do not pursue this direction here, the argument appears sufficiently robust to be applied in settings where the Fourier approach is not available.

The regularity of $u$, combined with the results of Chemin and Lerner \cite{CheminLerner1995}, yields the following dimensional bound for the advected patch.

\begin{corollary}\label{cor:Hdim}
    Let $(\rho, u)$ be the unique immediately strong solution corresponding to the initial data $(\rho_0, u_0)$ defined in \eqref{eq:initial_data2}. Then, the associated flow map $X$ satisfies
    \begin{align}\label{eq:Xeps}
        X \in L^\infty_{\loc}([0,\infty); C^{1-\eps}(\R^2)) \qquad \text{for all } \eps \in (0, 1).
    \end{align}
    In particular, the Hausdorff dimension of the patch boundary is preserved:
    \begin{align}\label{eq:Hdim}
        \dim_H(\partial D(t)) = 1 \qquad \text{for all } t \geq 0.
    \end{align}
\end{corollary}

\begin{proof}
    Thanks to\eqref{eq:logLip2}, the flow $X(t, \cdot)$ admits the following modulus of continuity for $|x-y|$ sufficiently small (see \cite[Theorem 1.3]{CheminLerner1995}):
    \begin{align}\label{eq:Xest}
        |X(t,x) - X(t,y)| \leq \exp \left( - \left[ (-\log |x-y|)^\eta - \eta \| \gamma \|_{L^1(0,t)} \right]^\frac{1}{\eta} \right).
    \end{align}
    Since $\eta > 0$, \eqref{eq:Xest} directly implies the Hölder regularity stated in \eqref{eq:Xeps}. 
    Consequently, the standard properties of the Hausdorff dimension under Hölder continuous mappings yield
    \begin{equation}
         \dim_H(\partial D(t)) \leq \frac{1}{1-\eps} \dim_H(\partial D_0) = \frac{1}{1-\eps}.
    \end{equation}
    Sending $\eps \to 0^+$ concludes the proof of \eqref{eq:Hdim}.
\end{proof}

Some remarks are in order:
\begin{enumerate}
    \item It is clear that the estimate \eqref{eq:Xeps} holds true in both scenarios of \cref{thm:B}. Moreover, any $1$-dimensional object is mapped by $X(t,\cdot)$ into a $1$-dimensional object.
    \item The condition $\eta > 0$ defines the threshold for the stability of the Hausdorff dimension of the boundary. In the classical log-Lipschitz case ($\eta = 0$), by taking the limit $\eta \to 0^+$ in \eqref{eq:Xest}, we recover the estimate
    \begin{align}\label{eq:Xloglip}
        |X(t,x) - X(t,y)| \leq |x-y|^{\alpha(t)}, \qquad \alpha(t) = \exp\left(- \| \gamma \|_{L^1(0,t)} \right).
    \end{align}
    This leads to a time-dependent bound on the Hausdorff dimension:
    \begin{align}
        \dim_H(\partial D(t)) \le \frac{1}{\alpha(t)} \dim_H(\partial D_0) = \exp\left(\| \gamma \|_{L^1(0,t)} \right).
    \end{align}
    Thus, the strictly positive assumption $\eta > 0$ is exactly what we need to guarantee that the boundary remains a $1$-dimensional object for all time $t \geq 0$.

  \item While our result ensures that the boundary does not grow dimensionally, it leaves open the question of its underlying geometric structure. As previously discussed, the lack of Lipschitz regularity (inherent to the $L^2$ regularity of the initial data) allows for the possibility of cusp formation, which precludes the preservation of the interface as a Lipschitz curve. In particular, it remains an open question whether the perimeter of the patch becomes infinite over time.
    
\end{enumerate}

\subsection{Strategy of the proof}

To prove \cref{thm:A}, our goal is to place the energy estimate for the difference of two solutions within a Gronwall framework. Let
$(\rho_1,u_1)$ and $(\rho_2,u_2)$ be two immediately strong solutions arising
from the same initial data, and denote
\begin{align}
\delta \rho := \rho_1-\rho_2,
\qquad
\delta u := u_1-u_2.
\end{align}
A formal computation yields
\begin{align}\label{eq:energy-difference}
    \frac12 \|\sqrt{\rho_1}\,\delta u(t)\|_{L^2}^2
    + \int_0^t \|\nabla \delta u(s)\|_{L^2}^2 \,\dd s
    \leq I(t)+J(t),
\end{align}
where
\begin{align}
I(t):=
\left|
\int_0^t \int_{\R^2}
\rho_1 \delta u \otimes \delta u : \nabla u_2
\right|, \qquad 
J(t):=
\left|
\int_0^t \int_{\R^2}
\delta \rho \,\dot u_2 \cdot \delta u
\right|.
\end{align}

We first discuss the term $I(t)$. Away from vacuum, the classical approach
follows the homogeneous setting and relies on Ladyzhenskaya's inequality:
\begin{align}\label{eq:oldineq}
    \|\sqrt{\rho_1}\,\delta u(s)\|_{L^4}^2
    \lesssim
    \|\delta u(s)\|_{L^2}
    \|\nabla \delta u(s)\|_{L^2}
    \lesssim
    \|\sqrt{\rho_1}\,\delta u(s)\|_{L^2}
    \|\nabla \delta u(s)\|_{L^2}.
\end{align}

In the $L^2$ vacuum setting, however, estimate~\eqref{eq:oldineq} is no longer
available, since the weighted norm
$\|\sqrt{\rho_1}\,\delta u\|_{L^2}$ only controls $\delta u$ on the support of
$\rho_1$. Moreover, the estimate cannot be closed through a direct Gronwall
argument because, in general,
\begin{align}
\nabla u_2 \notin L^1_tL^\infty_x.
\end{align}

The main idea is to replace \eqref{eq:oldineq} with a localized estimate on the
evolving patch. Using the displacement bound \eqref{eq:dispBound}, we prove
that
\begin{equation}\label{eq:goal}
\|\delta u(s)\|_{L^4(X_1(s,D))}^2
\lesssim
s^{1/2}\|\nabla \delta u(s)\|_{L^2}^2
+
\|\sqrt{\rho_1}\,\delta u(s)\|_{L^2}
\|\nabla \delta u(s)\|_{L^2}
+
\mathrm{Err}(s),
\end{equation}
where $X_1$ is the flow associated with $u_1$. Since
\begin{align}\label{eq:flow-L4}
    \|\sqrt{\rho_1}\,\delta u(s)\|_{L^4}
    =
    \|\delta u(s)\|_{L^4(X_1(s,D))},
\end{align}
estimate \eqref{eq:goal} plays the role of \eqref{eq:oldineq} in the vacuum
case. This is enough to estimate $I(t)$, after a suitable use of Young's and
Hölder's inequalities.

The term $J(t)$ is more delicate. It contains $\delta\rho$ rather than
$\rho_1$, and therefore it measures the mismatch between the two transported
patches. In particular, it is no longer sufficient to estimate $\delta u$ only
on $X_1(s,D)$. After passing to Lagrangian coordinates, one is naturally led to
control $\delta u$ on the family of sets
\begin{align}
    Y_\tau(s,D)
    :=
    X_2(s;\tau,X_1(\tau,D)),
    \qquad 0\leq \tau\leq s\leq t,
\end{align}
where $X_2(s;\tau,\cdot)$ denotes the flow associated with $u_2$ starting from
time $\tau$.

The key point is that the proof of \eqref{eq:goal} does not rely on any specific property of the flow $X_1$, but only on the fact that the underlying map is measure-preserving and satisfies the parabolic displacement estimate \eqref{eq:dispBound}.
Consequently, the same argument yields an estimate of the form
\eqref{eq:goal}, with $X_1(s,D)$ replaced by $Y_\tau(s,D)$, provided that
\begin{align}\label{eq:goal2}
    |Y_\tau(s,x)-x|
    \lesssim
    \sqrt{s},
    \qquad
    \text{for a.e. } x\in D.
\end{align}
In particular, this allows us to control $\delta u$ on the transported sets
$Y_\tau(s,D)$ in the same way as on $X_1(s,D)$.

If the flow $X_2$ were uniformly Lipschitz, then \eqref{eq:goal2} would follow
immediately. Indeed,
\begin{align}
|X_2(s;\tau,X_1(\tau,x))-x|
&\leq
|X_2(s;\tau,X_1(\tau,x))-X_2(s;\tau,x)|
+
|X_2(s;\tau,x)-x| \notag\\
&\leq
\|\nabla X_2\|_{L^\infty_{t,x}}
|X_1(\tau,x)-x|
+
|X_2(s;\tau,x)-x|
\lesssim
\sqrt{s}.
\end{align}
However, such a Lipschitz bound would require
\begin{align}
\nabla u_2 \in L^1_tL^\infty_x,
\end{align}
which is not available at the energy level.

In \cref{prop:composedflow}, we prove that the displacement estimate
\eqref{eq:goal2} nevertheless remains valid under the weaker assumption
\begin{align}
    \sqrt{t}\,\nabla u_2 \in L^2_tL^\infty_x,
\end{align}
which is satisfied by immediately strong solutions as noticed in \cite{HaoShaoWeiZhang2025}. Although this condition
does not yield uniform Lipschitz control of the flow, it still provides enough
quantitative control on the distortion of trajectories to propagate the
parabolic displacement estimate along the composed flow
$Y_\tau(s,\cdot)$.

This is the key mechanism allowing us to estimate $J(t)$ and ultimately close
the Gronwall argument.
\vspace{0.5cm}

Now we briefly explain the strategy of the proof of \cref{thm:B}. Let $(\rho,u)$ denote either the unique Leray--Hopf solution associated with \eqref{eq:initial_data1}, or the unique immediately strong solution associated with \eqref{eq:initial_data2}.

The first ingredient is an atomic decomposition of the initial datum in the spirit of \cite{LionsPeetre1964}:
\begin{align}\label{eq:realInt}
L^2(\R^2)
=
[\dot H^{-s}(\R^2),\dot H^s(\R^2)]_{1/2,2},
\qquad
s\in(0,1].
\end{align}
Hence we can decompose
\[
u_0=\sum_{j\in\Z}u_{j,0},
\]
where $\{u_{j,0}\}_{j\in\Z}$ is a sequence of divergence-free atoms satisfying
\begin{align}\label{eq:atomic-l2}
c_j
:=
2^{-j/2}\|u_{j,0}\|_{\dot H^s}
+
2^{j/2}\|u_{j,0}\|_{\dot H^{-s}},
\qquad
\|c_j\|_{\ell^2(\Z)}<\infty.
\end{align}
For each atom $u_{j,0}$, we consider its evolution under the flow generated by $(\rho,u)$, namely the solution $u_j$ of the linearized system
\begin{align}\label{eq:linearized-atom}
\begin{cases}
\partial_t(\rho u_j)
+\Div(\rho u_j\otimes u)
+\nabla P_j
\;=\;
\Delta u_j,
\\
\Div u_j
\;=\;
0,
\\
u_j|_{t=0}
\;=\;
u_{j,0}.
\end{cases}
\end{align}

The second key ingredient is the parabolic decay satisfied by each atom. More precisely, for every $s\in(0,1/2)$ there exists $C>0$ such that, for almost every $t>0$ and every $j\in\Z$,
\begin{align}\label{eq:atom-decay}
\begin{aligned}
\|\nabla u_j(t)\|_{L^2}
&\le
C\min\Bigl\{
t^{\frac s2-\frac12}\|u_{j,0}\|_{\dot H^s},
\,
t^{-\frac s2-\frac12}\|u_{j,0}\|_{\dot H^{-s}}
\Bigr\},
\\
\|\nabla u_j(t)\|_{L^\infty}
&\le
C\min\Bigl\{
t^{\frac s2-1}\|u_{j,0}\|_{\dot H^s},
\,
t^{-\frac s2-1}\|u_{j,0}\|_{\dot H^{-s}}
\Bigr\}.
\end{aligned}
\end{align}

These estimates were obtained in \cite{Danchin2025} in the strictly positive density case \eqref{eq:initial_data1}, and in \cite{SkondricViolini2026} for density patches \eqref{eq:initial_data2}, for solutions of \eqref{eq:linearized-atom}.

Combining \eqref{eq:atomic-l2} and \eqref{eq:atom-decay}, we prove the log-Lipschitz estimate \eqref{eq:logLip2}. This extends the arguments of \cite{Danchin2025,SkondricViolini2026}, where the same decay estimates were used to prove Lipschitz regularity in the critical Besov case $u_0\in\dot B^0_{2,1}$, corresponding to $\{c_j\}\in\ell^1(\Z)$. 

Fix $p>2/s$ and by Morrey's inequality
\begin{align}\label{eq:modulus-splitting}
|u(t,x)-u(t,y)|
\lesssim
|x-y|
\sum_{|j|\le N}\|\nabla u_j(t)\|_{L^\infty}
+
|x-y|^{1-\frac2p}
\sum_{|j|>N}\|\nabla u_j(t)\|_{L^p}.
\end{align}
Then, we define
\begin{align}
A(t)
:=
\sum_{j\in\Z}
(1+|j|)^{\eta-1}
\|\nabla u_j(t)\|_{L^\infty},
\qquad
B(t):=
\sum_{j\in\Z}
2^{\frac{|j|}{sp}}
(1+|j|)^{\eta-1}
\|\nabla u_j(t)\|_{L^p}.
\end{align}
Since $\eta\in(0,1)$ and $ r\mapsto 2^{-r/(sp)}(1+r)^{1-\eta}$ is eventually decreasing one can show the following estimates
\begin{align}
\sum_{|j|\le N}\|\nabla u_j(t)\|_{L^\infty}
\le
(1+N)^{1-\eta}A(t), \qquad 
\sum_{|j|>N}\|\nabla u_j(t)\|_{L^p}
\lesssim
2^{-N/(sp)}(1+N)^{1-\eta}B(t).
\end{align}
Combining these bounds with \eqref{eq:modulus-splitting}, we obtain
\begin{align}\label{eq:preloglip}
|u(t,x)-u(t,y)|
\lesssim
(1+N)^{1-\eta}
\Bigl(
|x-y|A(t)
+
|x-y|^{1-\frac2p}2^{-N/(sp)}B(t)
\Bigr).
\end{align}
Inspired by \cite{CheminLerner1995} we choose
\[
N:=\lfloor-\log_2|x-y|\rfloor,
\]
and using that $sp>2$, a straightforward computation yields (for $|x-y|<1$)
\begin{align}\label{eq:loglip-AB}
|u(t,x)-u(t,y)|
\lesssim
|x-y|
(-\log|x-y|)^{1-\eta}
\bigl(A(t)+B(t)\bigr).
\end{align}
Therefore, the main difficulty is to prove that
\[
A,B\in L^1_{\mathrm{loc}}([0,\infty)),
\]
assuming \eqref{eq:atom-decay}. This proof is is discussed in Section 4.

\section{Preliminary estimates and structural properties}
In this section, we collect the main structural properties of immediately strong solutions, together with stability estimates describing the interaction between two such solutions. These results will be used throughout the proof of \cref{thm:A}.

For convenience, we introduce the following two recurring settings.

\medskip

\newcounter{scenario}
\renewcommand{\thescenario}{S\arabic{scenario}}

\refstepcounter{scenario}\label{item:single_solution}
\textup{(\thescenario)} 
Let $(\rho,u)$ be an immediately strong solution with initial data $(\rho_0,u_0)$ satisfying \eqref{eq:initial_data2}, and let $X$ denote the associated flow. We write $C_u$ for the constant appearing in \eqref{eq:ImmStrongEst}.

\medskip

\refstepcounter{scenario}\label{item:two_solutions}
\textup{(\thescenario)} 
Let $(\rho_1,u_1)$ and $(\rho_2,u_2)$ be two immediately strong solutions corresponding to the initial data $(\rho_0,u_1(0))$ and $(\rho_0,u_2(0))$, respectively, satisfying \eqref{eq:initial_data2}. Let $X_1$ and $X_2$ denote the associated flows. We write $C_{u_1}$ and $C_{u_2}$ for the constants appearing in \eqref{eq:ImmStrongEst} associated with $u_1$ and $u_2$, respectively.
\subsection{Estimates along the flow}

\begin{lemma}\label{lem:cruc}
Let $T>0$. Under assumption \ref{item:single_solution} there exists $C=C(C_u,T,D)$ such that for every $f\in \dot H^1(\R^2) \cap L^2(D)$ and almost every $t\in (0,T)$,
\begin{equation}\label{eq:comp_L4}
\|f\|_{L^4(X(t,D))}
\le C \Big(
t^{1/4}\,\|\nabla f\|_{L^2(\R^2)}
+ \|f\|_{L^2(D)}^{1/2}\,\|\nabla f\|_{L^2(\R^2)}^{1/2}
+ \|f\|_{L^2(D)}
\Big).
\end{equation}
\end{lemma}

\begin{proof}
We denote by $C$ any constant depending only on $C_u$, $T$ and $D$, which may change from line to line.
We recall the fractional maximal operator of order $1/2$:
\begin{equation}
M_{1/2}(g)(x)
= \sup_{r>0} r^{-3/2}\int_{B_r(x)} |g(y)|\,\dd y.
\end{equation}
There exists a negligible set $N\subset\R^2$ such that for all $x,y\notin N$
(see \cite[Eq.~(1.2)]{KinnunenMartio1997})
\begin{equation}\label{eq:max}
|f(y)-f(x)|
\le C\,|y-x|^{1/2}
\Big((M_{1/2}\nabla f)(y)+(M_{1/2}\nabla f)(x)\Big).
\end{equation}
Applying \eqref{eq:max} with $y=X(t,x)$ and using the displacement estimate from \cref{thm:existL2} (see \eqref{eq:dispBound}),
\begin{equation}
|X(t,x)-x|\le C\sqrt{t} \qquad \text{for a.e. } x\in D, \qquad \text{for a.e. t>0} ,
\end{equation}
we obtain for almost every $x\in D$
\begin{equation}\label{eq:max_sq}
|f(X(t,x))-f(x)|^4
\le C t
\Big(
|(M_{1/2}\nabla f)(X(t,x))|^4
+|(M_{1/2}\nabla f)(x)|^4
\Big).
\end{equation}
Integrating over $D$ and using that $X(t,\cdot)$ is measure-preserving, we get
\begin{equation}
\|f\circ X(t,\cdot)-f\|_{L^4(D)}^4
\le C t \|M_{1/2}(\nabla f)\|_{L^4(\R^2)}^4.
\end{equation}
Taking the fourth root and using the boundedness $M_{1/2}:L^2\to L^4$
(see \cite[Section~1.1]{CaponeCruzUribeFiorenza2007}), we infer
\begin{equation}\label{eq:comp_L2}
\|f\circ X(t,\cdot)-f\|_{L^4(D)}
\le C\,t^{1/4}\,\|\nabla f\|_{L^2(\R^2)}.
\end{equation}
Since $X(t,\cdot)$ is measure-preserving,
\begin{equation}
\|f\|_{L^4(X(t,D))}=\|f\circ X(t,\cdot)\|_{L^4(D)},
\end{equation}
and therefore
\begin{equation}
\|f\|_{L^4(X(t,D))}
\le C\,t^{1/4}\,\|\nabla f\|_{L^2(\R^2)}+\|f\|_{L^4(D)}.
\end{equation}
Finally, by the two-dimensional Gagliardo--Nirenberg inequality,
\begin{equation}\label{eq:comp_L4D}
\|f\|_{L^4(D)}
\le C
\|f\|_{L^2(D)}^{1/2}\|\nabla f\|_{L^2(\R^2)}^{1/2}
+ C \|f\|_{L^2(D)}.
\end{equation}
Combining \eqref{eq:comp_L2} with \eqref{eq:comp_L4D} yields \eqref{eq:comp_L4}.
\end{proof}
\begin{remark}\label{rem:crucial}
Inspecting the previous proof, we observe that the estimate \eqref{eq:comp_L4} holds for any flow map $X(t)$ which is measure-preserving and satisfies
\begin{equation}
|X(t,x)-x|\le K \sqrt{t} \qquad \text{for a.e. } x\in D, \qquad \text{for a.e. } t\in (0,T).
\end{equation}
In this case, the constant $C$ depends only on $K$ and $D$.
\end{remark}

We will repeatedly use the following estimate, which is a simplified version of the result in \cite{SkondricViolini2026} (see Section $3$ therein).

\begin{lemma}\label{lem:Lpstrip}
Under assumption \ref{item:single_solution}, fix $R>0$ and $T>0$. Let
\begin{align}\label{eq:Drt}
    D_{R,t} := D + B_{R\sqrt{t}}(0).
\end{align}
Then there exists a constant $C=C(C_u,R,T,D)>0$ such that, for every measurable function $v$ and for almost every $t\in(0,T)$,
\begin{align}
\|v\|_{L^2(D_{R,t})}
&\le
C \Big(
t^{\frac12}\,\|\nabla v\|_{L^2(\R^2)}
+
\|\rho(t)\,v\|_{L^2(\R^2)}
\Big), \label{eq:strip_L2}\\
\|v\|_{L^\infty(D_{R,t})}
&\le
C \Big(
t^{\frac14}\,\|\nabla v\|_{L^4(\R^2)}
+
t^{-\frac12}\,\|\rho(t)\,v\|_{L^2(\R^2)}
\Big). \label{eq:strip_Linf}
\end{align}
\end{lemma}
\subsection{Flow distortion}
We now turn to the analysis of the distortion of the flow associated with immediately strong solutions. If $u_0 \notin \dot B^0_{2,1}(\R^2)$, the velocity field is no longer Lipschitz, and the associated flow may exhibit a loss of regularity. The next result provides a quantitative control of this effect.

\begin{proposition}\label{prop:composedflow}
Under assumption \ref{item:two_solutions}, fix $T>0$. For $0<s<t<T$, define
\begin{equation}
Y_s(t,x):=X_2\bigl(t,s, X_1(s,x)\bigr).
\end{equation}
Then there exists a constant $C=C(C_{u_1},C_{u_2},T,D)>0$ such that, for almost every $x\in D$ and almost every $0<s<t<T$,
\begin{equation}
|Y_s(t,x)-x|\le C\sqrt{t}.
\end{equation}
\end{proposition}
\begin{proof}
We first estimate the Lipschitz constant of the flow $X_2$. By the standard bound,
\begin{align}
\|\nabla X_2(t,s,\cdot)\|_{L^\infty(\R^2)}
\le \exp\!\left(\int_s^t \|\nabla u_2(\tau)\|_{L^\infty(\R^2)}\,d\tau\right).
\end{align}
Moreover, by \cite[Equation 3.11]{HaoShaoWeiZhang2026} and Young's inequality,
\begin{align}
\|\nabla u_2(\tau)\|_{L^\infty}^2 
&\lesssim
\|\sqrt{\rho} \dot u_2(\tau)\|_{L^2}^2
+
\|\nabla u_2(\tau)\|_{L^2}\,\|\nabla \dot u_2(\tau)\|_{L^2} \\
&\lesssim
\|\sqrt{\rho} \dot u_2(\tau)\|_{L^2}^2
+ \tau^{-1}\|\nabla u_2(\tau)\|_{L^2}^2
+ \tau \|\nabla \dot u_2(\tau)\|_{L^2}^2.
\end{align}
Multiplying by $\tau$ and integrating over $(0,\infty)$, we obtain
\begin{align}
\int_0^\infty \tau \|\nabla u_2(\tau)\|_{L^\infty}^2\,d\tau
\lesssim C_{u_2}^2.
\end{align}
Hence, for every $0<s<t$, by Cauchy--Schwarz,
\begin{align}
\int_s^t \|\nabla u_2(\tau)\|_{L^\infty}\,d\tau
&\le
\left(\int_s^t \tau^{-1}\,d\tau\right)^{1/2}
\left(\int_s^t \tau \|\nabla u_2(\tau)\|_{L^\infty}^2\,d\tau\right)^{1/2}
\le C_{u_2}\sqrt{\log\frac{t}{s}}.
\end{align}
Therefore,
\begin{align}\label{eq:almostlipflow}
\|\nabla X_2(t,s,\cdot)\|_{L^\infty(\R^2)}
&\le \exp\!\Big(C_{u_2}\sqrt{\log\tfrac{t}{s}}\Big)
\le C \sqrt{\frac{t}{s}},
\end{align}
where $C>0$ depends only on $C_{u_2}$, and we used the elementary inequality
\begin{align}
a\sqrt{x}\le \tfrac12 x + \tfrac{a^2}{2}, \qquad x\ge 0.
\end{align}
Fix $x\in D$. Using the flow property
\begin{align}
X_2(t,x)=X_2\bigl(t,s,X_2(s,x)\bigr),
\end{align}
we write
\begin{align}
|Y_s(t,x)-x|
\le 
|X_2(t,s,X_1(s,x))-X_2(t,s,X_2(s,x))|
+
|X_2(t,x)-x|.
\end{align}
By \cref{thm:existL2}, there exist constants $C_i=C_i(C_{u_i},T,D)$ such that
\begin{align}
|X_i(\tau,x)-x|\le C_i\sqrt{\tau}, \qquad i=1,2,
\end{align}
for all $\tau\ge0$ and a.e.~$x\in D$. In particular,
\begin{align}
|X_2(t,x)-x|\le C_2\sqrt{t},
\qquad
|X_1(s,x)-X_2(s,x)|
\le (C_1+C_2)\sqrt{s}.
\end{align}
Using \eqref{eq:almostlipflow}, we estimate
\begin{align}
|X_2(t,s,X_1(s,x))-X_2(t,s,X_2(s,x))|
\le
\|\nabla X_2(t,s,\cdot)\|_{L^\infty}\,|X_1(s,x)-X_2(s,x)|
\le C(C_1+C_2)\sqrt{t}.
\end{align}
Combining the above estimates yields the claim.
\end{proof}

The quantity
\begin{align}\label{eq:qLL}
w_M(t,s) := \exp\!\left( M \sqrt{\log \frac{t}{s}} \right)
\end{align}
will play a central role in the sequel. As a consequence of \cref{prop:composedflow}, we obtain the following estimate.

\begin{corollary}\label{cor:contrdist}
Under assumption \textup{(\ref{item:single_solution})}, there exists a constant $M=M(C_u)>0$ such that, for almost every $0<s<t$,
\begin{equation}\label{eq:contrdist}
\|\nabla X(t,s,\cdot)\|_{L^\infty(\R^2)}
\le w_M(t,s).
\end{equation}
\end{corollary}

\begin{proof}
The claim follows directly from \eqref{eq:almostlipflow}.
\end{proof}

A closely related weight appeared in \cite{HaoShaoWeiZhang2025} in a different context. In particular, the authors establish a variant of the following result.

\begin{lemma}\label{lem:weighted}
Let $M>0$ and define
\begin{align}
L_{w_M} f(t) := \frac{1}{t} \int_0^t w_M(t,s)\, f(s)\, \dd s.
\end{align}
Then $L_{w_M}$ is bounded on $L^p(0,\infty)$ for every $p\in(1,\infty]$. Moreover, its operator norm depends only on $p$ and $M$. In particular, taking $p=\infty$ and $f\equiv 1$, there exists a constant $C=C(M)>0$ such that for every $t>0$,
\begin{align}
\int_0^t w_M(t,s)\,\dd s \le C\,t.
\end{align}
\end{lemma}

\begin{proof}
After the change of variables $s=t\tau$, we write
\begin{align}
L_{w_M} f(t)
= \int_0^1 \exp\!\big( M\sqrt{|\log \tau|} \big)\, f(t\tau)\,\dd \tau.
\end{align}
By Minkowski’s integral inequality, for every $p\in(1,\infty]$,
\begin{align}
\|L_{w_M} f\|_{L^p(0,T)}
\le
\int_0^1 \exp\!\big( M\sqrt{|\log \tau|} \big)\,
\|f(t\tau)\|_{L^p(0,T)}\,\dd \tau.
\end{align}
A change of variables yields
\begin{align}
\|f(t\tau)\|_{L^p(0,T)}
= \tau^{-1/p}\|f\|_{L^p(0,\tau T)},
\end{align}
and therefore
\begin{align}
\|L_{w_M} f\|_{L^p(0,T)}
\le
\|f\|_{L^p(0,T)}
\int_0^1 \exp\!\big( M\sqrt{|\log \tau|} \big)\, \tau^{-1/p}\,\dd \tau.
\end{align}
Setting $\tau=e^{-y}$, we obtain
\begin{align}
\int_0^1 \exp\!\big( M\sqrt{|\log \tau|} \big)\, \tau^{-1/p}\,\dd \tau
=
\int_0^\infty e^{M\sqrt{y}} e^{-(1-1/p)y}\,\dd y,
\end{align}
which is finite for every $p>1$, since the exponential decay dominates the sublinear growth $e^{M\sqrt{y}}$.
\end{proof}

\subsection{Duality formulation}

We now exploit the distortion control provided by \eqref{eq:contrdist} to quantify the loss of regularity in time for the transport equation.

\begin{proposition}\label{prop:duality}
Under assumption \textup{(S2)}, let $\delta \rho := \rho_1 - \rho_2$ and $\delta u := u_1 - u_2$. Then $\delta \rho$ satisfies
\begin{align}\label{eq:negstab}
\begin{cases}
\partial_t \delta \rho + \Div (\delta \rho\, u_2) = -\Div (\rho_1\, \delta u), \\
\delta \rho(0)=0.
\end{cases}
\end{align}
Moreover, for every $t>0$ and every $g=g_1 g_2$, one has
\begin{align}
\int_{\R^2}\delta \rho(t,x)\,g(x)\,\dd x
=
\int_0^t\int_{\R^2} \rho_1\, \delta u (s,x)\cdot \nabla\bigl(\varphi_1(s,x)\,\varphi_2(s,x)\bigr)\,\dd x\,\dd s,
\end{align}
where, for $i=1,2$, $\varphi_i$ solves
\begin{align}\label{eq:duals}
\begin{cases}
\partial_s\varphi_i + u_2\cdot\nabla\varphi_i =0,\\
\varphi_i(t)=g_i.
\end{cases}
\end{align}
Furthermore, for every $p\in[1,\infty]$, $i=1,2$, and every measurable set $A\subset \R^2$, we have
\begin{align}
\|\varphi_i(s)\|_{L^p(A)}
&\le
\|g_i\|_{L^p(X_2(t,s,A))},\\
\|\nabla \varphi_i(s)\|_{L^p(\R^2)}
&\le
w_{M_2}(t,s)\,\|\nabla g_i\|_{L^p(\R^2)},
\end{align}
where $M_2>0$ is the constant associated with $u_2$ in \cref{cor:contrdist}.
\end{proposition}
\begin{proof}
The proof is based on duality. Fix $t>0$, and let $\varphi$ solve the backward transport equation on $(0,t)$:
\begin{align}
\begin{cases}
\partial_s\varphi + u_2\cdot\nabla\varphi =0,\\
\varphi(t)=g_1 g_2.
\end{cases}
\end{align}
Since $\Div u_2=0$, a standard duality argument yields
\begin{align}
\int_{\R^2}\delta \rho(t,x)\,g(x)\,\dd x
=
\int_0^t\int_{\R^2} \rho_1 \delta u(s,x)\cdot \nabla\varphi(s,x)\,\dd x\,\dd s.
\end{align}
Moreover, by uniqueness of solutions to the transport equation, we have the factorization
\begin{align}
\varphi = \varphi_1\,\varphi_2,
\end{align}
where $\varphi_i$ solves \eqref{eq:duals} for $i=1,2$. For the $L^p$ estimate, using that $X_2(t,s,\cdot)$ is measure-preserving, we obtain for $i=1,2$
\begin{align}
\|\varphi_i(s)\|_{L^p(A)}
\le
\|\varphi_i(t)\|_{L^p(X_2(t,s,A))}
=
\|g_i\|_{L^p(X_2(t,s,A))}.
\end{align}
For the gradient, the standard flow estimate gives
\begin{align}
\|\nabla\varphi_i(s)\|_{L^{p}(\R^2)}
\le
\|\nabla g_i\|_{L^{p}(\R^2)}
\|\nabla X_2(t,s,\cdot)\|_{L^\infty(\R^2)}
\leq
w_{M_2}(t,s)\,\|\nabla g_i\|_{L^{p}(\R^2)}.
\end{align}
This concludes the proof.
\end{proof}
\section{Proof of \cref{thm:A}}\label{sec:Sec4}
The proof proceeds in steps, utilizing a relative energy argument inspired by \cite{CrinBaratSkondricViolini2025} to show that two solutions with identical initial data coincide. Since the low-regularity regime prevents directly using one solution as a test function for the other, we follow \cite{Skondric2025} by regularizing one solution via mollification and passing to the limit at the end. This technical setup is covered in Step 1; readers familiar with the technique may start directly from Step 2.

\subsection*{Step 1.}
Assume that $(\rho,u)$ is an immediately strong solution arising from initial data $(\rho_0,u_0)$. Then, by definition,
\begin{equation}
\sup_{i\in\{0,1,2,3\}} A_i(\rho,u) \le C_u.
\end{equation}
Let $u_{0,n} := u_0 \ast \varphi_n$, where $\varphi_n$ is a standard mollifier. By \cref{thm:existL2} (and specifically \cite[Theorem 2.2]{SkondricViolini2026}), there exists an immediately strong solution $(\rho_n,u_n)$ with initial data $(\rho_0,u_{0,n})$. Moreover, this solution satisfies
\begin{equation}
\sup_{i\in\{0,1,2,3\}} A_i(\rho_n,u_n) \le C_{u_n},
\end{equation}
where $C_{u_n}$ depends non-decreasingly on $\|\sqrt{\rho_0}\,u_{0,n}\|_{L^2(\R^2)}$. Since $\rho_0 \le 1$, we have 
\begin{equation}
\|\sqrt{\rho_0}\,u_{0,n}\|_{L^2(\R^2)}
\le
\|u_{0,n}\|_{L^2(\R^2)}
\le 
\|u_0\|_{L^2(\R^2)}.
\end{equation}
Consequently, there exists a uniform constant $\widetilde{C}>0$, independent of $n$, such that
\begin{equation}
\sup_{i\in\{0,1,2,3\}} A_i(\rho_n,u_n) \le \widetilde{C}.
\end{equation}

Fix $T>0$. Since $(\rho_n,u_n)$ is associated with a smooth initial velocity, the relative energy inequality holds for almost every $t\in(0,T)$ (see \cite[Lemma 4.1]{CrinBaratSkondricViolini2025}):
\begin{equation}\label{eq:relatenergyn}
\begin{aligned}
\|\sqrt{\rho}\,&(u-u_n)(t)\|_{L^2(\R^2)}^2
+ \int_0^t \|\nabla (u-u_n)(s)\|_{L^2(\R^2)}^2 \,\mathrm{d}s
\le
\|\sqrt{\rho_0}\,(u_0-u_{0,n})\|_{L^2(\R^2)}^2 \\
&\quad + \int_0^t \left| \int_{\R^2} \rho\, (u-u_n)\otimes (u-u_n): \nabla u_n \,\mathrm{d}x \right| \mathrm{d}s + \int_0^t \left| \int_{\R^2} (\rho - \rho_n)\, \dot u_n \cdot (u-u_n)\, \mathrm{d}x \right| \mathrm{d}s.
\end{aligned}
\end{equation}

Our main goal is to establish the existence of a constant $C = C(C_u,\widetilde{C},T,D)>0$ such that, for almost every $t\in(0,T)$,
\begin{equation}\label{eq:thesis}
\|\sqrt{\rho}\,(u-u_n)(t)\|_{L^2(\R^2)}^2
+ \int_0^t \|\nabla (u-u_n)(s)\|_{L^2(\R^2)}^2 \,\mathrm{d}s
\le
C\, \|\sqrt{\rho_0}\,(u_0-u_{0,n})\|_{L^2(\R^2)}^2.
\end{equation}
We postpone the proof of \eqref{eq:thesis} to the subsequent steps and show first how it implies uniqueness.

To upgrade the control provided by \eqref{eq:thesis} to a local $L^2$-convergence, we combine it with \cref{lem:Lpstrip} and a Poincar\'e-type argument. By applying \cref{lem:Lpstrip} with $R=0$, we find a constant $C=C(C_u,T,D)>0$ such that for every measurable function $v$ and every $t\in(0,T)$,
\begin{equation}\label{eq:uL2DD}
\|v\|_{L^2(D)}
\le
C \left(
t^{1/2}\|\nabla v\|_{L^2(\R^2)}
+
\|\sqrt{\rho(t)}\,v\|_{L^2(\R^2)}
\right).
\end{equation}
Fix $R_0>0$ such that $D\subset B_{R_0}(0)$, and let $R>R_0$. For any open set $\Omega$, denote the integral average by $(v)_{\Omega} := |\Omega|^{-1} \int_{\Omega} v(x)\,\mathrm{d}x$. By Poincar\'e's inequality on $B_R$, we have
\begin{equation}
\|v\|_{L^2(B_R)}
\lesssim_R
\|\nabla v\|_{L^2(B_R)}
+
|(v)_{B_R}|.
\end{equation}
Moreover, adding and subtracting the average over $D$, we estimate
\begin{equation}
|(v)_{B_R}|
\le |(v)_{B_R} - (v)_D| + |(v)_D|
\lesssim_{R,D}
\|\nabla v\|_{L^2(B_R)}
+
\|v\|_{L^2(D)}.
\end{equation}
Combining these estimates with \eqref{eq:uL2DD}, we obtain for $t\in(0,T)$ the bound
\begin{equation}\label{eq:Extpoincare}
\|v\|_{L^2(B_R)}
\lesssim_{C_u,T,D,R}
\|\nabla v\|_{L^2(\R^2)}
+
\|\sqrt{\rho(t)}\,v\|_{L^2(\R^2)}.
\end{equation}
Since $u_{0,n} \to u_0$ in $L^2$, the right-hand side of \eqref{eq:thesis} converges to zero. Therefore, evaluating \eqref{eq:thesis} yields
\begin{equation}
\lim_{n\to \infty} \sup_{t \in (0,T)} \|\sqrt{\rho(t)}\,(u(t)-u_n(t))\|_{L^2(\R^2)}^2 = 0,
\qquad
\lim_{n\to \infty} \int_0^T \|\nabla (u-u_n)(t)\|_{L^2(\R^2)}^2 \,\mathrm{d}t = 0.
\end{equation}
Applying \eqref{eq:Extpoincare} to the difference $v=(u-u_n)(t)$ and integrating over $(0,T)$, we conclude
\begin{equation}
\lim_{n\to \infty} \int_0^T \|(u-u_n)(t)\|_{L^2(B_R)}^2 \,\mathrm{d}t = 0,
\end{equation}
which implies
\begin{equation}
u_n \to u \quad \text{in } L^2(0,T;L^2_{\loc}(\R^2)).
\end{equation}
Finally, if $(\tilde\rho,\tilde u)$ is another immediately strong solution originating from the same initial data, the exact same procedure yields
\begin{equation}
u_n \to \tilde u \quad \text{in } L^2(0,T;L^2_{\loc}(\R^2)).
\end{equation}
By uniqueness of the limit, $u=\tilde u$ almost everywhere in $(0,T)\times\R^2$. Since $T>0$ was chosen arbitrarily, we deduce
\begin{equation}
u=\tilde u \qquad \text{a.e. in } (0,\infty)\times\R^2.
\end{equation}
Consequently, $\rho=\tilde\rho$ almost everywhere. This follows from the uniqueness of bounded solutions to the transport equation with a divergence-free Sobolev velocity field, as established by the DiPerna-Lions theory \cite{DipLi89}.
\subsection*{Step 2.}

We now turn to the proof of \eqref{eq:thesis}. To simplify the notation, we relabel the solutions as follows.

We denote the immediately strong solution $(\rho,u)$ by $(\rho_1,u_1)$, with associated constant $C_{u_1}$. The mollified solution $(\rho_n,u_n)$ is denoted by $(\rho_2,u_2)$, with associated constant $C_{u_2}$ (previously denoted by $\widetilde{C}$).

We set
\begin{align}
\delta u := u_1 - u_2,
\qquad
\delta \rho := \rho_1 - \rho_2.
\end{align}
Moreover, we introduce
\begin{align}
h(t) := \sup_{s \in (0,t)} \|\sqrt{\rho_1(s)}\,\delta u(s)\|_{L^2(\R^2)},
\qquad
d(t) := \|\nabla \delta u(t)\|_{L^2(\R^2)}.
\end{align}

In what follows, we do not track explicitly the dependence of constants on $C_{u_1}$, $C_{u_2}$, $T$, and $D$. We write $\lesssim$ to denote inequalities up to such constants, and use $C$ to denote generic constants depending only on these quantities. Any additional dependence will be specified explicitly.

We will repeatedly use that, by \cref{thm:existL2}, there exists a constant $N=N(C_{u_1},C_{u_2},T,D)>0$ such that, for all $t \ge 0$,
\begin{align}
\supp \rho_i(t) \subset D + B_{N\sqrt{t}}(0) =: D_{N,t},
\qquad i=1,2.
\end{align}
With this notation, the relative energy inequality \eqref{eq:relatenergyn} reads
\begin{align}\label{eq:relatenergy12}
h(T)^2
+ \int_0^T d(t)^2 \,\dd t
\le
h(0)^2
+ \int_0^T P(t)\,\dd t
+ \int_0^T N(t)\,\dd t,
\end{align}
where
\begin{align}
P(t)
&:= \int_{\R^2} \rho_1(t,x)\, \delta u(t,x) \otimes \delta u(t,x): \nabla u_2(t,x)\, \dd x, \\
N(t)
&:= \int_{\R^2} \delta \rho(t,x)\, \dot u_2(t,x) \cdot \delta u(t,x)\, \dd x.
\end{align}
We now estimate the terms $P(t)$ and $N(t)$ and close \eqref{eq:relatenergy12} by Grönwall’s inequality.

\subsection*{Step 3.}

We start with $P(t)$. By \cref{lem:cruc}, in particular \eqref{eq:comp_L4} with $X=X_1$, for a.e.~$t\in(0,T)$ and every $f$,
\begin{align}
\|f\|_{L^4(X_1(t,D))}
\lesssim
t^{1/4}\|\nabla f\|_{L^2(\R^2)}
+ \|f\|_{L^2(D)}^{1/2}\|\nabla f\|_{L^2(\R^2)}^{1/2}
+ \|f\|_{L^2(D)}.
\end{align}
Since $\rho_1(t)=\mathbf{1}_{X_1(t,D)}$, this yields
\begin{align}
\|\rho_1(t)f\|_{L^4(\R^2)}=\|f\|_{L^4(X_1(t,D))}.
\end{align}
Combining with \cref{lem:Lpstrip} (with $R=0$),
\begin{align}
\|f\|_{L^2(D)}
\lesssim
t^{1/2}\|\nabla f\|_{L^2(\R^2)}
+
\|\rho_1(t)f\|_{L^2(\R^2)},
\end{align}
we obtain
\begin{align}
\|\rho_1(t)f\|_{L^4(\R^2)}
\lesssim
t^{1/4}\|\nabla f\|_{L^2(\R^2)}
+ \|\rho_1(t)f\|_{L^2}^{1/2}\|\nabla f\|_{L^2}^{1/2}
+ \|\rho_1(t)f\|_{L^2}.
\end{align}
Applying this with $f=\delta u(t)$ gives
\begin{align}\label{eq:fourest}
\|\rho_1\delta u(t)\|_{L^4(\R^2)}
\lesssim
t^{1/4}d(t)
+ h(t)^{1/2}d(t)^{1/2}
+ h(t).
\end{align}
By Hölder’s inequality,
\begin{align}
P(t)
\le
\|\sqrt{\rho_1}\delta u(t)\|_{L^4}
\,
\|\sqrt{\rho_1}\delta u(t)\,\nabla u_2(t)\|_{L^{4/3}}.
\end{align}
Using \eqref{eq:fourest}, we write
\begin{align}
P(t) \lesssim P_1(t)+P_2(t)+P_3(t),
\end{align}
where
\begin{align}
P_1(t)
&= t^{1/4}d(t)
\|\sqrt{\rho_1}\delta u(t)\,\nabla u_2(t)\|_{L^{4/3}}, \\
P_2(t)
&= h(t)^{1/2}d(t)^{1/2}
\|\sqrt{\rho_1}\delta u(t)\,\nabla u_2(t)\|_{L^{4/3}}, \\
P_3(t)
&= h(t)
\|\sqrt{\rho_1}\delta u(t)\,\nabla u_2(t)\|_{L^{4/3}}.
\end{align}
For $P_1$, by Hölder and Young, for any $\eps>0$
\begin{align}
P_1(t)
\le
t^{1/4}d(t)
\|\sqrt{\rho_1}\delta u(t)\|_{L^2}
\|\nabla u_2(t)\|_{L^4}
\le
\varepsilon d(t)^2
+ C_\varepsilon h(t)^2 t^{1/2}\|\nabla u_2(t)\|_{L^4}^2.
\end{align}
For $P_2$, using again \eqref{eq:fourest} and Young,
\begin{align}
P_2(t)
&\lesssim
h(t)^{1/2}d(t)^{1/2}
\left(
t^{1/4}d(t)
+ h(t)^{1/2}d(t)^{1/2}
+ h(t)
\right)
\|\nabla u_2(t)\|_{L^2} \\
&\le
\varepsilon d(t)^2
+ C_\varepsilon h(t)^2
\left(
t\|\nabla u_2(t)\|_{L^2}^4
+ \|\nabla u_2(t)\|_{L^2}^2
+ \|\nabla u_2(t)\|_{L^2}^{4/3}
\right).
\end{align}
For $P_3$,
\begin{align}
P_3(t)
\le
h(t)^2 \|\nabla u_2(t)\|_{L^4}.
\end{align}
Collecting the above estimates, we conclude that for every $\varepsilon>0$ and almost every $t \in (0,T)$
\begin{align}\label{eq:Pest}
P(t)
\le
\varepsilon d(t)^2
+ C_\varepsilon h(t)^2 \gamma_1(t),
\end{align}
where
\begin{align}\label{eq:gamma1}
\gamma_1(t)
:=\;
t^{1/2}\|\nabla u_2(t)\|_{L^4}^2
+ t\|\nabla u_2(t)\|_{L^2}^4
+ \|\nabla u_2(t)\|_{L^2}^2
+ \|\nabla u_2(t)\|_{L^2}^{4/3}
+ \|\nabla u_2(t)\|_{L^4}.
\end{align}
\subsection*{Step 4.}

We now estimate the second term $N(t)$. By \cref{prop:duality}, applied with
\begin{align}
g_1 = \dot u_2(t), \qquad g_2 = \delta u(t),
\end{align}
we obtain
\begin{align}
N(t)
=
\int_0^t \int_{\R^2} \rho_1(s,x)\, \delta u(s,x)\cdot \nabla\bigl(\varphi_1(s,x)\,\varphi_2(s,x)\bigr)\,\dd x\,\dd s,
\end{align}
where, for $i=1,2$, $\varphi_i$ solves the backward transport equation associated with $g_i$.

Moreover, for every $p\in[1,\infty]$, $i=1,2$, and every measurable set $A\subset \R^2$, we have
\begin{align}
\begin{aligned}\label{eq:phi1}
\|\varphi_1(s)\|_{L^p(A)}
&\le
\|\dot u_2(t)\|_{L^p(X_2(t,s,A))}, \\
\|\nabla \varphi_1(s)\|_{L^p(\R^2)}
&\le
w_{M_2}(t,s)\,\|\nabla \dot u_2(t)\|_{L^p(\R^2)},
\end{aligned}
\end{align}
and
\begin{align}
\begin{aligned}\label{eq:phi2}
\|\varphi_2(s)\|_{L^p(A)}
&\le
\|\delta u(t)\|_{L^p(X_2(t,s,A))}, \\
\|\nabla \varphi_2(s)\|_{L^p(\R^2)}
&\le
w_{M_2}(t,s)\,\|\nabla \delta u(t)\|_{L^p(\R^2)}.
\end{aligned}
\end{align}
Here $M_2>0$ is the constant associated with $u_2$ in \cref{cor:contrdist}, and $w_{M_2}$ is defined in \eqref{eq:qLL}.

We decompose
\begin{align}
N(t) \le N_1(t) + N_2(t),
\end{align}
where
\begin{align}
N_1(t)
=
\left| \int_0^t \int_{\R^2} \rho_1\, \delta u \cdot \varphi_1 \nabla \varphi_2 \,\dd x\, \dd s \right|, \quad
N_2(t)=
\left| \int_0^t \int_{\R^2} \rho_1\, \delta u \cdot \varphi_2 \nabla \varphi_1 \,\dd x\, \dd s \right|.
\end{align}
We begin with $N_1(t)$. By Hölder's inequality, the fact that $\rho_1(s)$ is supported on $X_1(s,D)$, and \eqref{eq:phi1}--\eqref{eq:phi2}, we obtain
\begin{align}
N_1(t)
&\le
\int_0^t
\|\sqrt{\rho_1}\,\delta u(s)\|_{L^2(\R^2)}
\|\varphi_1(s)\|_{L^\infty(X_1(s,D))}
\|\nabla \varphi_2(s)\|_{L^2(\R^2)} \,\dd s \\
&\le
\int_0^t
h(s)\,
\|\dot u_2(t)\|_{L^\infty(X_2(t,s,X_1(s,D)))}
\, w_{M_2}(t,s)\, d(t)\,\dd s.
\end{align}
Using the definition of $h$, we infer
\begin{align}
N_1(t)
\le
d(t)\, h(t)
\int_0^t
\|\dot u_2(t)\|_{L^\infty(X_2(t,s,X_1(s,D)))}
\, w_{M_2}(t,s)\, \dd s.
\end{align}
By \cref{prop:composedflow}, the composed flow
\begin{align}
Y_s(t,x):=X_2\bigl(t,s,X_1(s,x)\bigr)
\end{align}
satisfies, for a.e.~$t\in(0,T)$ and $x\in D$,
\begin{align}\label{eq:Yprop}
|Y_s(t,x)-x|\le \hat{C}\sqrt{t},
\qquad
\text{hence }
\quad
X_2(t,s,X_1(s,D)) \subset D_{\hat{C},t}.
\end{align}
Therefore,
\begin{align}
\int_0^t
\|\dot u_2(t)\|_{L^\infty(X_2(t,s,X_1(s,D)))}
\, w_{M_2}(t,s)\,\dd s
\le
\|\dot u_2(t)\|_{L^\infty(D_{\hat{C},t})}
\int_0^t w_{M_2}(t,s)\,\dd s.
\end{align}
By \cref{lem:weighted},
\begin{align}\label{eq:w2weigh}
\int_0^t w_{M_2}(t,s)\,\dd s \lesssim t,
\end{align}
and thus
\begin{align}
N_1(t)
\lesssim
d(t)\, h(t)\, t \|\dot u_2(t)\|_{L^\infty(D_{\hat{C},t})}.
\end{align}
By Young's inequality,
\begin{align}\label{eq:N1est}
N_1(t)
\le
\varepsilon d(t)^2
+
C_\varepsilon h(t)^2\, t^2 \|\dot u_2(t)\|_{L^\infty(D_{\hat{C},t})}^2.
\end{align}
We now turn to $N_2(t)$. By Hölder's inequality and \eqref{eq:phi2}, we write
\begin{align}
N_2(t)
\le
\int_0^t
\|\rho_1 \delta u(s)\nabla \varphi_1(s)\|_{L^{4/3}}
\|\delta u(t)\|_{L^4(Y_s(t,D))} \,\dd s.
\end{align}
Using \cref{lem:cruc} and, in particular, \cref{rem:crucial}, we can apply the estimate \eqref{eq:comp_L4} with the flow $Y_s$ and $f=\delta u(t)$. Since \eqref{eq:Yprop} ensures that $Y_s$ satisfies the required displacement bound, we obtain
\begin{equation}
\|\delta u(t)\|_{L^4(Y_s(t,D))}
\le C \Big(
t^{1/4}\|\nabla \delta u(t)\|_{L^2(\R^2)}
+ \|\delta u(t)\|_{L^2(D)}^{1/2}\|\nabla \delta u(t)\|_{L^2(\R^2)}^{1/2}
+ \|\delta u(t)\|_{L^2(D)}
\Big).
\end{equation}
Arguing as in \eqref{eq:fourest}, we deduce
\begin{align}\label{eq:fourest2}
\|\delta u(t)\|_{L^4(Y_s(t,D))}
\lesssim
t^{1/4}d(t)
+ h(t)^{1/2}d(t)^{1/2}
+ h(t).
\end{align}
Inserting \eqref{eq:fourest2}, we decompose
\begin{align}
N_2(t) \lesssim N_{2,1}(t)+N_{2,2}(t)+N_{2,3}(t),
\end{align}
with the obvious definitions.

\emph{Estimate of $N_{2,1}$.}
By Hölder, \eqref{eq:phi1} and \eqref{eq:w2weigh}
\begin{align}
N_{2,1}(t)
&\lesssim
t^{1/4}d(t)
\int_0^t
\|\rho_1 \delta u(s)\|_{L^2}
\|\nabla \varphi_1(s)\|_{L^4}\,\dd s \lesssim
t^{1/4}d(t) h(t)
\|\nabla \dot u_2(t)\|_{L^4}
\int_0^t w_{M_2}(t,s)\,\dd s \\
&\lesssim
t^{5/4}d(t) h(t)
\|\nabla \dot u_2(t)\|_{L^4}.
\end{align}
Hence,
\begin{align}
N_{2,1}(t)
\le
\varepsilon d(t)^2
+
C_\varepsilon h(t)^2\, t^{5/2}
\|\nabla \dot u_2(t)\|_{L^4}^2.
\end{align}

\emph{Estimate of $N_{2,3}$.}
Arguing similarly,
\begin{align} 
N_{2,3}(t) \leq h(t) \int_0^t
\|\rho_1 \delta u(s)\nabla \varphi_1(s)\|_{L^{4/3}}
 \,\dd s
\lesssim
h(t)^2\, t \|\nabla \dot u_2(t)\|_{L^4}.
\end{align}

\emph{Estimate of $N_{2,2}$.}
Using Hölder and \eqref{eq:phi1},
\begin{align}
N_{2,2}(t) &\lesssim  h(t)^{1/2}d(t)^{1/2}
 \int_0^t
\|\rho_1 \delta u(s)\|_{L^{4}} \|\nabla \varphi_1(s)\|_{L^{2}}
\,\dd s. \\ & \lesssim
h(t)^{1/2}d(t)^{1/2}
\|\nabla \dot u_2(t)\|_{L^2}
\int_0^t w_{M_2}(t,s)
\|\rho_1 \delta u(s)\|_{L^{4}}\,\dd s.
\end{align}
Applying \eqref{eq:fourest}, we obtain
\begin{align}
N_{2,2}(t)
\lesssim
h(t)^{1/2}d(t)^{1/2}
\|\nabla \dot u_2(t)\|_{L^2}
\int_0^t w_{M_2}(t,s)
\Big(
t^{1/4}d(s)
+ h(s)^{1/2}d(s)^{1/2}
+ h(s)
\Big)\dd s.
\end{align}
Define
\begin{align}
(L_w g)(t) := \frac{1}{t}\int_0^t w_{M_2}(t,s)\,g(s)\,\dd s.
\end{align}
Then, again by \eqref{eq:w2weigh},
\begin{align}
\int_0^t w_{M_2}(t,s)(\cdots)\,\dd s
\lesssim
t^{5/4}(L_wd )(t)
+ t h(t)^{1/2}(L_w d^{1/2})(t)
+ t h(t).
\end{align}
Using Young's inequality,
\begin{align}
N_{2,2}(t)
\le
\varepsilon d(t)^2
+ \varepsilon (L_wd(t))^2
+ \varepsilon (L_wd^{1/2}(t))^4
+ C_\varepsilon h(t)^2 \gamma_2(t),
\end{align}
where
\begin{align}\label{eq:gamma2}
\gamma_2(t)
:=
t^{5/2}\|\nabla \dot u_2(t)\|_{L^4}^2
+ t \|\nabla \dot u_2(t)\|_{L^4}
+ t^5 \|\nabla \dot u_2(t)\|_{L^2}^4
+ t^2 \|\nabla \dot u_2(t)\|_{L^2}^2
+ t^{4/3} \|\nabla \dot u_2(t)\|_{L^2}^{4/3}.
\end{align}
Collecting estimates for $N_{2,1}$, $N_{2,2}$ and $N_{2,3}$, we conclude
\begin{align}\label{eq:N2est}
N_2(t)
\le
\varepsilon d(t)^2
+ \varepsilon (L_wd(t))^2
+ \varepsilon (L_wd^{1/2}(t))^4
+ C_\varepsilon h(t)^2 \gamma_2(t).
\end{align}
\subsection*{Step 5.}

Collecting the estimates for $P$, $N_1$, and $N_2$ (see \eqref{eq:Pest}, \eqref{eq:N1est}, and \eqref{eq:N2est}), we obtain
\begin{align}
P(t) + N_1(t) + N_2(t)
\le
\varepsilon d(t)^2
+ \varepsilon (L_w d(t))^2
+ \varepsilon (L_w d^{1/2}(t))^4
+ h(t)^2\, \gamma(t),
\end{align}
where
\begin{align}
\gamma(t)
:=
\gamma_1(t)
+ t^2 \|\dot u_2(t)\|_{L^\infty(D_{\hat{C},t})}^2
+ \gamma_2(t).
\end{align}
Assuming that $\gamma \in L^1(0,T)$, inserting this estimate into \eqref{eq:relatenergy12} yields
\begin{align}\label{eq:relatenergy3}
h(T)^2
+ \int_0^T d(t)^2 \,\dd t
&\le
h(0)^2
+ \varepsilon \int_0^T d(t)^2 \dd t
+ \varepsilon \int_0^T (L_w d(t))^2 \dd t
 \\ & + \varepsilon \int_0^T (L_w d^{1/2}(t))^4 \dd t 
+ C_\varepsilon \int_0^T h(t)^2\, \gamma(t)\,\dd t.
\end{align}
Since $L_w$ is bounded on $L^p(0,T)$ (see \cref{lem:weighted}), we have
\begin{align}
\int_0^T (L_w d(t))^2 \dd t
\le C \int_0^T d(t)^2 \dd t,
\end{align}
and similarly
\begin{align}
\int_0^T (L_w d^{1/2}(t))^4 \dd t
\le C \int_0^T d(t)^2 \dd t.
\end{align}
Therefore,
\begin{align}
h(t)^2
+ \int_0^T d(t)^2 \dd t
\le
h(0)^2
+ \varepsilon(1+2C) \int_0^T d(t)^2 \dd t
+ C_\varepsilon \int_0^T h(t)^2\, \gamma(t)\,\dd t.
\end{align}
Choosing $\varepsilon>0$ sufficiently small, we can absorb the second term on the left-hand side and obtain
\begin{align}\label{eq:relatenergy4}
h(T)^2
+ \frac{1}{2}\int_0^T d(t)^2 \dd t
\le
h(0)^2
+ C_\varepsilon \int_0^T h(t)^2\, \gamma(t)\,\dd t.
\end{align}
An application of Grönwall's inequality yields
\begin{align}
h(T)^2
+ \frac{1}{2}\int_0^T d(t)^2 \dd t
\le
h(0)^2 \exp\!\left( \int_0^T \gamma(t)\,\dd t \right).
\end{align}
To conclude \eqref{eq:thesis}, it remains to show that $\gamma \in L^1(0,T)$.
\subsection*{Step 6.}

We first show that $\gamma_2 \in L^1(0,T)$, arguing term by term. For the first contribution, using Gagliardo--Nirenberg and Young's inequality, we write
\begin{align}\label{eq:bigBound}
\begin{aligned}
\int_0^T t^{5/2}\|\nabla \dot u_2(t)\|_{L^4}^2 \dd t
&\lesssim
\int_0^T t^{5/2}
\|\nabla \dot u_2(t)\|_{L^2}
\|\nabla^2 \dot u_2(t)\|_{L^2} \dd t \\
&\le
\int_0^T t^2\|\nabla \dot u_2(t)\|_{L^2}^2 \dd t
+
\int_0^T t^{3}\|\nabla^2 \dot u_2(t)\|_{L^2}^2 \dd t
\le C.
\end{aligned}
\end{align}
For the second term, using interpolation and Young's inequality,
\begin{align}
\int_0^T t \|\nabla \dot u_2(t)\|_{L^4} \dd t
&\lesssim
\int_0^T t^{-1/4}
\bigl(t\|\nabla \dot u_2(t)\|_{L^2}\bigr)^{1/2}
\bigl(t^{3/2}\|\nabla^2 \dot u_2(t)\|_{L^2}\bigr)^{1/2}
\dd t \\
&\lesssim
\int_0^T t^{-1/2} \dd t
+
\int_0^T t^2\|\nabla \dot u_2(t)\|_{L^2}^2 \dd t
+
\int_0^T t^3\|\nabla^2 \dot u_2(t)\|_{L^2}^2 \dd t
\le C.
\end{align}
For the third term, we use a supremum bound:
\begin{align}
\int_0^T t^5 \|\nabla \dot u_2(t)\|_{L^2}^4 \dd t
\lesssim
\sup_{t\in(0,T)} t^3 \|\nabla \dot u_2(t)\|_{L^2}^2
\int_0^T t^2 \|\nabla \dot u_2(t)\|_{L^2}^2 \dd t
\le C.
\end{align}
The fourth contribution is directly controlled by the energy bounds. For the last term, by Hölder's inequality,
\begin{align}
\int_0^T t^{4/3} \|\nabla \dot u_2(t)\|_{L^2}^{4/3} \dd t
\lesssim
\int_0^T t^2 \|\nabla \dot u_2(t)\|_{L^2}^2 \dd t
\le C.
\end{align}
This proves that $\gamma_2 \in L^1(0,T)$.

The strategy for $\gamma_1$ is analogous. It remains to show that
\begin{align}
\int_0^T t^2  \|\dot u_2(t)\|_{L^\infty(D_{\hat{C},t})}^2 \dd t \le C.
\end{align}
To this end, we use \cref{lem:Lpstrip}, which ensures that for every $v$ and $t\in(0,T)$,
\begin{align}
\|v\|_{L^\infty(D_{\hat{C},t})}
\le
C (
t^{\frac14}\|\nabla v\|_{L^4(\R^2)}
+t^{-\frac12}
\|\rho_2(t)\,v\|_{L^2(\R^2)}),
\end{align}
where $C$ depends on$C_{u_2}, T, D$ and  $\hat{C}$ $C_{u_2}$. Choosing $v = \dot u_2$, we obtain
\begin{align}
\int_0^T t^2 \|\dot u_2(t)\|_{L^\infty(D_{\hat{C},t})}^2 \dd t
\lesssim
\int_0^T t^{5/2}\|\nabla \dot u_2(t)\|_{L^4}^2 \dd t
+
\int_0^T t \|\rho_2(t)\dot u_2(t)\|_{L^2}^2 \dd t.
\end{align}
The second term is controlled by the energy bounds, while the first one has already been estimated in \eqref{eq:bigBound}. 

This concludes the proof that $\gamma \in L^1(0,T)$.

\section{Proof of \cref{thm:B}}

\textbf{Step A.} In this step, we recall several results from the literature that will be used throughout the proof.

Let $(\rho,u)$ be the unique immediately strong solution of \cref{thm:B} associated with initial data $(\rho_0,u_0)$ satisfying \eqref{eq:initial_data2}. Fix
\[
s\in \Bigl(0,\frac12\Bigr).
\]
By real interpolation (see \eqref{eq:realInt}), the initial velocity can be decomposed as
\[
u_0=\sum_{j\in\Z}u_{j,0},
\]
where $\{u_{j,0}\}_{j\in\Z}$ is a sequence of divergence-free atoms satisfying
\begin{align}\label{eq:atomic}
c_j
:=
2^{-j/2}\|u_{j,0}\|_{\dot H^s}
+
2^{j/2}\|u_{j,0}\|_{\dot H^{-s}},
\qquad
\|c_j\|_{\ell^2(\Z)}<\infty.
\end{align}
For each $j\in\Z$, consider the linearized system
\begin{align}\label{eq:linearizej}
\begin{cases}
\partial_t(\rho u_j)
+\Div(\rho u_j\otimes u)
+\nabla P
=
\Delta u_j,
\\
\Div u_j =0,
\\
u_j|_{t=0}=u_{j,0}.
\end{cases}
\end{align}

It is proved in \cite[Section~4 and Section~5]{SkondricViolini2026} that there exists a solution $u_j$ to \eqref{eq:linearizej} such that, for some constant $C>0$ depending only on $D$ and $\|u_0\|_{L^2}$, but independent of $j$ and $t>0$, the following estimates hold:
\begin{align}\label{eq:L2jdecays}
\begin{aligned}
\|\nabla u_j(t)\|_{L^2}
&\le
C\min\Bigl\{
t^{-1/2+s/2}\|u_{j,0}\|_{\dot H^s},
\,
t^{-1/2-s/2}\|u_{j,0}\|_{\dot H^{-s}}
\Bigr\},
\\
\|\sqrt{\rho}\,\dot u_j(t)\|_{L^2}
&\le
C\min\Bigl\{
t^{-1+s/2}\|u_{j,0}\|_{\dot H^s},
\,
t^{-1-s/2}\|u_{j,0}\|_{\dot H^{-s}}
\Bigr\},
\\
\|\nabla \dot u_j(t)\|_{L^2}
&\le
C\min\Bigl\{
t^{-3/2+s/2}\|u_{j,0}\|_{\dot H^s},
\,
t^{-3/2-s/2}\|u_{j,0}\|_{\dot H^{-s}}
\Bigr\},
\end{aligned}
\end{align}
where
\[
\dot u_j
=
\partial_tu_j+(u\cdot\nabla)u_j
\]
denotes the material derivative.

We also make use of the scale-invariant inequality introduced in \cite[Equation~(3.11)]{HaoShaoWeiZhang2026}, which allows one to control the Lipschitz norm using only $L^2$-based quantities:
\begin{align}\label{eq:Lipnorm}
\|\nabla u_j(t)\|_{L^\infty}
\lesssim
\|\sqrt{\rho}\,\dot u_j(t)\|_{L^2}
+
\|\nabla u_j(t)\|_{L^2}^{1/2}
\|\nabla \dot u_j(t)\|_{L^2}^{1/2}.
\end{align}
Combining \eqref{eq:L2jdecays} with \eqref{eq:Lipnorm}, we obtain
\begin{align}\label{eq:Linftyjdecays}
\|\nabla u_j(t)\|_{L^\infty}
\le
C\min\Bigl\{
t^{-1+s/2}\|u_{j,0}\|_{\dot H^s},
\,
t^{-1-s/2}\|u_{j,0}\|_{\dot H^{-s}}
\Bigr\}.
\end{align}
Since the constant $C$ is independent of $j$, interpolation between the first estimate in \eqref{eq:L2jdecays} and \eqref{eq:Linftyjdecays} yields, for every $p\in(2,\infty)$,
\begin{align}\label{eq:uj-Lp-min}
\|\nabla u_j(t)\|_{L^p}
\le
C\min\Bigl\{
t^{\frac{s}{2}+\frac1p-1}\|u_{j,0}\|_{\dot H^s},
\,
t^{-\frac{s}{2}+\frac1p-1}\|u_{j,0}\|_{\dot H^{-s}}
\Bigr\}.
\end{align}
Finally, we recall that
\begin{align}\label{eq:Dconv}
\sum_{|j|\le J}u_j
\to u
\qquad
\text{in }
\mathcal{D}'\bigl((0,\infty)\times\R^2\bigr)
\quad
\text{as }
J\to\infty,
\end{align}
which is proved in \cite[Section~4]{SkondricViolini2026}.

\medskip

\textbf{Step B.} Let $p>\frac{2}{s}$,
and define
\begin{align}
A(t)
:=
\sum_{j\in\Z}
(1+|j|)^{\eta-1}
\|\nabla u_j(t)\|_{L^\infty},
\qquad
B(t)
:=
\sum_{j\in\Z}
2^{\frac{|j|}{sp}}
(1+|j|)^{\eta-1}
\|\nabla u_j(t)\|_{L^p}.
\end{align}
Fix $T>0$. Throughout the proof, all implicit constants may depend on
\[
C,s,p,\eta,\text{ and }T.
\]
Using \eqref{eq:Dconv}, together with Morrey's inequality, we obtain that for every $N>0$,
\begin{align}
\begin{aligned}\label{eq:split-morrey-thm}
|u(t,x)-u(t,y)|
&\le
\sum_{|j|\le N}
|u_j(t,x)-u_j(t,y)|
+
\sum_{|j|>N}
|u_j(t,x)-u_j(t,y)|
\\
&\lesssim
|x-y|
\sum_{|j|\le N}
\|\nabla u_j(t)\|_{L^\infty}
+
|x-y|^{1-\frac2p}
\sum_{|j|>N}
\|\nabla u_j(t)\|_{L^p}.
\end{aligned}
\end{align}
Since $\eta\in(0,1)$, we have
\begin{align}\label{eq:lowfreq-thm}
\sum_{|j|\le N}\|\nabla u_j(t)\|_{L^\infty}
\le (1+N)^{1-\eta}
\sum_{|j|\le N}(1+|j|)^{\eta-1}\|\nabla u_j(t)\|_{L^\infty}
\le (1+N)^{1-\eta} A(t).
\end{align}
Moreover, since the function
\begin{align}
r\mapsto 2^{-r/(sp)}(1+r)^{1-\eta}
\end{align}
is eventually decreasing, we have
\begin{align}
2^{-\frac{|j|}{sp}}(1+|j|)^{1-\eta}
\lesssim
2^{-\frac{N}{sp}}(1+N)^{1-\eta}
\qquad
\text{for all } |j|>N.
\end{align}
Therefore,
\begin{align}
\begin{aligned}\label{eq:highfreq-thm}
\sum_{|j|>N}\|\nabla u_j(t)\|_{L^p}
&=
\sum_{|j|>N}
2^{-\frac{|j|}{sp}}(1+|j|)^{1-\eta}
\Bigl(
2^{\frac{|j|}{sp}}(1+|j|)^{\eta-1}
\|\nabla u_j(t)\|_{L^p}
\Bigr) \\
&\lesssim
2^{-\frac{N}{sp}}(1+N)^{1-\eta} B(t).
\end{aligned}
\end{align}
Combining \eqref{eq:lowfreq-thm}, \eqref{eq:highfreq-thm} with \eqref{eq:split-morrey-thm}, we obtain
\begin{align}\label{eq:pre-final}
|u(t,x)-u(t,y)|
\lesssim
(1+N)^{1-\eta}
\Big(
|x-y|\, A(t)
+
|x-y|^{1-\frac2p}2^{-\frac{N}{sp}} B(t)
\Big).
\end{align}
We now choose
\begin{align}
N:=\bigl\lfloor -\log_2 |x-y| \bigr\rfloor.
\end{align}
Then
\begin{align}
2^{-N}\le 2|x-y|.
\end{align}
With this choice of $N$, \eqref{eq:pre-final} yields
\begin{align}
|u(t,x)-u(t,y)|
&\lesssim
|x-y|\,(-\log |x-y|)^{1-\eta} A(t) \\
&\quad+
|x-y|^{1-\frac2p+\frac1{sp}}
(-\log |x-y|)^{1-\eta} B(t).
\nonumber
\end{align}
Since $s<1\slash 2$, we have
\begin{align}
1-\frac2p+\frac1{sp}>1.
\end{align}
Hence, recalling that $0<|x-y|<1$,
\begin{align}
|x-y|^{1-\frac2p+\frac1{sp}}
\le |x-y|.
\end{align}
Therefore,
\begin{align}\label{eq:AB-bound}
|u(t,x)-u(t,y)|
\lesssim
|x-y|\,(-\log |x-y|)^{1-\eta}\bigl(A(t)+B(t)\bigr).
\end{align}

\medskip
\textbf{Step C.} We are left to show that $A,B\in L^1(0,T)$. To this end, we claim that for every $q\in[2,\infty]$ with $s>2/q$,
\begin{align}\label{eq:unified}
\int_0^T \|\nabla u_j(t)\|_{L^q}\,dt
\lesssim 2^{-\frac{|j|}{sq}}\,c_j,
\qquad j\in\Z,
\end{align}
with the convention $1/\infty=0$.

Assuming \eqref{eq:unified}, we obtain
\begin{align}
\|B\|_{L^1(0,T)}
&\le \sum_{j\in\Z} 2^{\frac{|j|}{sp}}(1+|j|)^{\eta-1}
\int_0^T \|\nabla u_j(t)\|_{L^p}\,dt 
\lesssim \sum_{j\in\Z}(1+|j|)^{\eta-1}c_j.
\nonumber
\end{align}
Since $\eta<\frac12$, Cauchy--Schwarz yields
\begin{align}
\sum_{j\in\Z}(1+|j|)^{\eta-1}c_j
&\le
\Big(\sum_{j\in\Z}(1+|j|)^{2\eta-2}\Big)^{1/2}
\|c_j\|_{\ell^2(\Z)}
\lesssim
\|c_j\|_{\ell^2(\Z)}.
\end{align}
Hence,
\begin{align}
\|B\|_{L^1(0,T)} \lesssim \|c_j\|_{\ell^2(\Z)}.
\end{align}
Similarly, applying \eqref{eq:unified} with $q=\infty$, we obtain
\begin{align}
\|A\|_{L^1(0,T)}
\lesssim \sum_{j\in\Z}(1+|j|)^{\eta-1}c_j
\lesssim \|c_j\|_{\ell^2(\Z)}.
\end{align}
Therefore $A,B\in L^1(0,T)$.

It remains to prove \eqref{eq:unified}. Fix $j\in\Z$, and set $\tau:=2^{-j/s}$. Using \eqref{eq:uj-Lp-min}, we split the time integral:
\begin{align}
\int_0^T \|\nabla u_j(t)\|_{L^q}\,dt
&\lesssim
\int_0^\tau t^{\frac s2+\frac1q-1}\|u_{j,0}\|_{\dot H^s}\,dt
+
\int_\tau^\infty t^{-\frac s2+\frac1q-1}\|u_{j,0}\|_{\dot H^{-s}}\,dt \\
&\lesssim
\tau^{\frac s2+\frac1q}\|u_{j,0}\|_{\dot H^s}
+
\tau^{-\frac s2+\frac1q}\|u_{j,0}\|_{\dot H^{-s}} \\
&=
2^{-j\left(\frac12+\frac{1}{sq}\right)}\|u_{j,0}\|_{\dot H^s}
+
2^{j\left(\frac12-\frac{1}{sq}\right)}\|u_{j,0}\|_{\dot H^{-s}}.
\end{align}
Hence
\begin{align}
\int_0^T \|\nabla u_j(t)\|_{L^q}\,dt
\lesssim 2^{-j/(sq)}\,c_j.
\end{align}
If $j\ge0$, then $-j/(sq)=-|j|/(sq)$, which gives \eqref{eq:unified}. If $j<0$, we instead use only the first branch:
\begin{align}
\int_0^T \|\nabla u_j(t)\|_{L^q}\,dt
&\lesssim
\int_0^T t^{\frac s2+\frac1q-1}\|u_{j,0}\|_{\dot H^s}\,dt
\lesssim
\|u_{j,0}\|_{\dot H^s}
\lesssim
2^{j/2}c_j.
\end{align}
Since $j<0$ and $sq>2$, we have $j/2\le j/(sq)$, hence
\begin{align}
2^{j/2}\le 2^{j/(sq)}=2^{-|j|/(sq)}.
\end{align}
This proves \eqref{eq:unified}.

Finally, define
\begin{align}
\gamma(t):=C\bigl(A(t)+B(t)\bigr),
\end{align}
where $C>0$ is the implicit constant in \eqref{eq:AB-bound}. Since $A,B\in L^1(0,T)$ and $T>0$ is arbitrary, we have $\gamma\in L^1_{\loc}([0,\infty))$, and \eqref{eq:logLip2} follows from \eqref{eq:AB-bound}.

\subsection*{Acknowledgments}

The author is supported by the Deutsche Forschungsgemeinschaft (DFG) through the project \emph{Inhomogeneous and compressible fluids: statistical solutions and dissipative anomalies} within the SPP 2410 \emph{Hyperbolic Balance Laws in Fluid Mechanics: Complexity, Scales, Randomness} (CoScaRa).

The author thanks G. Crippa and S. \v{S}kondri\'c for helpful discussions. The author is also grateful to S. \v{S}kondri\'c for an early review of this manuscript.

\noindent \textbf{Data availability statement} \,
 Data sharing not applicable to this article as no data sets were generated or analyzed during the current study.

\noindent \textbf{Conflicts of interest} The authors have no competing interests to declare that are relevant to the content of this
article.

\bibliographystyle{plain} 
\bibliography{biblio}

@article{Abidi2007,
  author  = {H. Abidi},
  title   = {Équation de Navier–Stokes avec densité et viscosité variables dans l’espace critique},
  journal = {Rev. Mat. Iberoam.},
  volume  = {23},
  year    = {2007},
  pages   = {537--586}
}

@article{AbidiGuiZhang2012,
  author  = {H. Abidi and G. Gui and P. Zhang},
  title   = {On the well-posedness of 3-D inhomogeneous Navier–Stokes equations in the critical spaces},
  journal = {Arch. Ration. Mech. Anal.},
  volume  = {204},
  year    = {2012},
  pages   = {189--230}
}

@article{DanchinMuchaPiasecki2024,
  author  = {R. Danchin and P. B. Mucha and T. Piasecki},
  title   = {Stability of the density patches problem with vacuum for incompressible inhomogeneous viscous flows},
  journal = {Ann. Inst. H. Poincaré Anal. Non Linéaire},
  volume  = {41},
  year    = {2024},
  pages   = {897--931}
}

@article{CrinBaratDeNittiSkondricViolini2025,
  author  = {T. Crin-Barat and N. De Nitti and S. \v{S}kondri\'c and A. Violini},
  title   = {Regularity aspects of Leray--Hopf solutions to the 2D inhomogeneous Navier--Stokes system and applications to weak-strong uniqueness},
  journal = {Nonlinearity},
  volume  = {38},
  pages   = {125020},
  year    = {2025}
}

@article{Danchin2003,
  author  = {R. Danchin},
  title   = {Density-dependent incompressible viscous fluids in critical spaces},
  journal = {Proc. Roy. Soc. Edinburgh Sect. A},
  volume  = {133},
  year    = {2003},
  pages   = {1311--1334}
}

@article{DanchinMucha2012,
  author  = {R. Danchin and P. B. Mucha},
  title   = {A Lagrangian approach for the incompressible Navier–Stokes equations with variable density},
  journal = {Comm. Pure Appl. Math.},
  volume  = {65},
  year    = {2012},
  pages   = {1458--1480}
}

@article{DanchinMucha2019,
  author  = {R. Danchin and P. B. Mucha},
  title   = {The incompressible Navier–Stokes equations in vacuum},
  journal = {Comm. Pure Appl. Math.},
  volume  = {72},
  year    = {2019},
  pages   = {1351--1385}
}

@article{GancedoGarcia2018,
  author  = {F. Gancedo and E. García-Juárez},
  title   = {Global regularity of 2D density patches for inhomogeneous Navier–Stokes},
  journal = {Arch. Ration. Mech. Anal.},
  volume  = {229},
  year    = {2018},
  pages   = {339--360}
}

@article{LionsPeetre1964,
  author = {J.-L. Lions and J. Peetre},
  title = {Sur une classe d'espaces d'interpolation},
  journal = {Publications Math{\'e}matiques de l'IH{\'E}S},
  volume = {19},
  year = {1964},
  pages = {5--68}
}

@article{PaicuZhangZhang2013,
  author  = {M. Paicu and P. Zhang and Z. Zhang},
  title   = {Global unique solvability of inhomogeneous Navier–Stokes equations with bounded density},
  journal = {Comm. Partial Differential Equations},
  volume  = {38},
  year    = {2013},
  pages   = {1208--1234}
}

@article{DipLi89,
  author  = {R. J. DiPerna and P.-L. Lions},
  title   = {Ordinary differential equations, transport theory and {S}obolev spaces},
  journal = {Invent. Math.},
  volume  = {98},
  year    = {1989},
  pages   = {511--547}
}

@article{Danchin2025,
  author  = {R. Danchin},
  title   = {Global well-posedness for two-dimensional inhomogeneous viscous flows with rough data via dynamic interpolation},
  journal = {Analysis \& PDE},
  volume  = {18},
  year    = {2025},
  pages   = {1231--1270}
}

@article{CrinBaratSkondricViolini2025,
  author    = {T. Crin-Barat and S. \v{S}kondri\'c and A. Violini},
  title     = {Relative energy method for weak–strong uniqueness of the inhomogeneous Navier–Stokes equations far from vacuum},
  journal   = {Journal of Evolution Equations},
  year      = {2025},
  volume    = {25},
  pages = {010001}
}

@article{HaoShaoWeiZhang2025,
  author  = {T. Hao and F. Shao and D. Wei and Z. Zhang},
  title   = {Global Well-Posedness of Inhomogeneous Navier--Stokes Equations with Bounded Density},
  journal = {International Mathematics Research Notices},
  volume  = {2025},
  year    = {2025},
  pages   = {rnaf283},
}

@article{HaoShaoWeiZhang2026,
  author  = {T. Hao and F. Shao and D. Wei and Z. Zhang},
  title   = {On the density patch problem for the 2-D inhomogeneous Navier--Stokes equations},
  journal = {Science China Mathematics},
  year    = {2026},
  doi     = {10.1007/s11425-025-2505-y}
}

@article{KinnunenMartio1997,
  author  = {J. Kinnunen and O. Martio},
  title   = {Hardy's Inequalities for Sobolev Functions},
  journal = {Mathematical Research Letters},
  volume  = {4},
  pages   = {489--500},
  year    = {1997}
}

@article{CaponeCruzUribeFiorenza2007,
  author  = {C. Capone and D. Cruz-Uribe and A. Fiorenza},
  title   = {The Fractional Maximal Operator and Fractional Integrals on Variable $L^p$ Spaces},
  journal = {Revista Matem{\'a}tica Iberoamericana},
  volume  = {23},
  pages   = {743--770},
  year    = {2007}
}

@article{CheminLerner1995,
  author  = {J.-Y. Chemin and N. Lerner},
  title   = {Flot de champs de vecteurs non-lipschitziens et {\'e}quations de Navier--Stokes},
  journal = {Journal of Differential Equations},
  volume  = {121},
  year    = {1995},
  pages   = {314--328},
}

@book{BCD2011,
  author    = {H. Bahouri and J.-Y. Chemin and R. Danchin},
  title     = {Fourier Analysis and Nonlinear Partial Differential Equations},
  publisher = {Springer},
  year      = {2011}
}

@book{Lions1996,
  author    = {P.-L. Lions},
  title     = {Mathematical Topics in Fluid Mechanics. Vol. 1. Incompressible Models},
  publisher = {Oxford University Press},
  year      = {1996}
}

@misc{GancedoGarciaJuarezLunaVelasco2025,
  author = {F. Gancedo and E. García-Juárez and P. Luna-Velasco},
  title  = {On 2D Navier--Stokes free boundary: nonnegative density and small viscosity contrast},
  note   = {arXiv:2507.09333},
  url    = {https://arxiv.org/abs/2507.09333},
  year   = {2025}
}

@misc{PrangeTan2023,
  author = {C. Prange and J. Tan},
  title  = {Free boundary regularity of vacuum states for incompressible viscous flows in unbounded domains},
  note   = {arXiv:2310.09288},
  url    = {https://arxiv.org/abs/2310.09288},
  year   = {2023}
}

@misc{Skondric2025,
  author = {S. \v{S}kondri\'c},
  title  = {Existence and uniqueness of Leray-Hopf weak solution for the inhomogeneous 2D Navier--Stokes equations without vacuum},
  note   = {arXiv:2504.16638},
  url    = {https://arxiv.org/abs/2504.16638},
  year   = {2025}
}

@misc{SkondricViolini2026,
  author = {S. {\v{S}}kondri{\'c} and A. Violini},
  title  = {On Lions' density patch problem at a critical level of regularity},
  note   = {arXiv:2604.16017},
  url    = {https://arxiv.org/abs/2604.16017},
  year   = {2026}
}

\end{document}